\documentclass[12pt]{article}
\usepackage{tikz,caption,geometry,amsthm,amssymb,amsmath,enumerate,float,cite,tikz}
\usetikzlibrary{decorations.pathreplacing}
\newtheorem{theorem}{Theorem}
\newtheorem{lemma}[theorem]{Lemma}

\newtheorem{claim}{Claim}

\allowdisplaybreaks
\usepackage{lineno}
\usepackage{setspace}

\title{On the maximum number of maximum\\ independent sets in connected graphs}
\author{E. Mohr \and D. Rautenbach}
\date{}

\begin{document}

\onehalfspace 

\maketitle

\begin{center}
Institut f\"{u}r Optimierung und Operations Research,
Universit\"{a}t Ulm, Ulm, Germany\\
\{\texttt{elena.mohr, dieter.rautenbach}\}\texttt{@uni-ulm.de}\\[3mm]
\end{center}

\begin{abstract}
We characterize the connected graphs of given order $n$ and given independence number $\alpha$ 
that maximize the number of maximum independent sets.
For $3\leq \alpha\leq n/2$,
there is a unique such graph that
arises from the disjoint union of $\alpha$ cliques of orders
$\left\lceil\frac{n}{\alpha}\right\rceil$
and 
$\left\lfloor\frac{n}{\alpha}\right\rfloor$,
by selecting a vertex $x$ in a largest clique
and adding an edge between $x$ and a vertex in each of the remaining $\alpha-1$ cliques.
Our result confirms a conjecture of Derikvand and Oboudi 
[On the number of maximum independent sets of graphs, Transactions on Combinatorics 3 (2014) 29-36].
\end{abstract}

\bigskip

\section{Introduction}

Moon and Moser's \cite{momo} classical result 
on the number of maximal cliques 
immediately yields a characterization of the graphs 
of a given order
that have the maximum number of maximum independent sets.
Similarly, 
the characterization of the connected graphs 
of a given order with that property
follows from a result of Griggs, Grinstead, and Guichard \cite{grgrgu};
see \cite{joch}.
Using a result of Zykov \cite{zy} allows to characterize 
the graphs of a given order and a given independence number
that have the maximum number of maximum independent sets;
see Theorem \ref{theorem1} below.
Our contribution in the present paper 
is the connected version of this result;
that is,
we characterize 
the connected graphs of a given order and a given independence number
that have the maximum number of maximum independent sets.
Our results confirm a recent conjecture of Derikvand and Oboudi \cite{deob}.

We consider only finite, simple, and undirected graphs,
and use standard terminology and notation. 
An {\it independent set} in a graph $G$
is a set of pairwise non-adjacent vertices of $G$.
The {\it independence number} $\alpha(G)$ of $G$ 
is the maximum cardinality of an independent set in $G$.
An independent set in $G$ is {\it maximum} if it has cardinality $\alpha(G)$.
For a graph $G$, 
let $\sharp\alpha(G)$ be the number of maximum independent sets in $G$. 
For a vertex $u$ of $G$, 
let $\sharp\alpha(G,u)$ be the number of maximum independent sets 
in $G$ that contains $u$. 

Let $n$ and $\alpha$ be positive integers with $\alpha<n$.

Let the graph $G(n,\alpha)$ 
be the disjoint union of 
one clique $C_0$ 
of order
$\left\lceil\frac{n}{\alpha}\right\rceil$,
and $\alpha-1$ further cliques 
$C_1,\ldots,C_{\alpha-1}$ 
of orders 
$\left\lceil\frac{n}{\alpha}\right\rceil$
and 
$\left\lfloor\frac{n}{\alpha}\right\rfloor$,
that is, the graph
$G(n,\alpha)$ is the complement of the {\it Tur\'{a}n graph} 
of order $n$ and clique number $\alpha$.
Let the graph $F(n,\alpha)$ arise from $G(n,\alpha)$
by adding the edges $x_0x_1,\ldots,x_0x_{\alpha-1}$,
where $x_i$ is a vertex in $C_i$ for every $i$ in $\{ 0,\ldots,\alpha-1\}$.
We will call the vertex $x_0$ the {\it special cutvertex} of $F(n,\alpha)$.
Note that $x_0$ may not be unique if $\alpha\leq 2$, and $n$ is a multiple of $\alpha$.

For $\frac{n}{\alpha}\geq 2$, let
$$\mathcal{F}(n,\alpha)=\begin{cases} 
\big\{F(n,\alpha),C_5\big\} &\text{, if } (n,\alpha)=(5,2),\mbox{ and}\\[3mm]
\big\{F(n,\alpha)\big\} &\text{, otherwise,}
\end{cases}$$
where $C_5$ denotes the cycle of order $5$, and
for $\frac{n}{\alpha}<2$, 
let $\mathcal{F}(n,\alpha)$ be the set of all connected graphs $G$ 
that have a vertex $x_0$ 
such that $G-x_0$ is isomorphic to $G(n-1,\alpha)$.
It is easy to see that every graph in $\mathcal{F}(n,\alpha)$
for $\frac{n}{\alpha}<2$
is isomorphic to a graph that arises from $F(n,\alpha)$
by possibly adding further edges incident with the special cutvertex $x_0$ of $F(n,\alpha)$.

See Figure \ref{fig1} for an illustration.

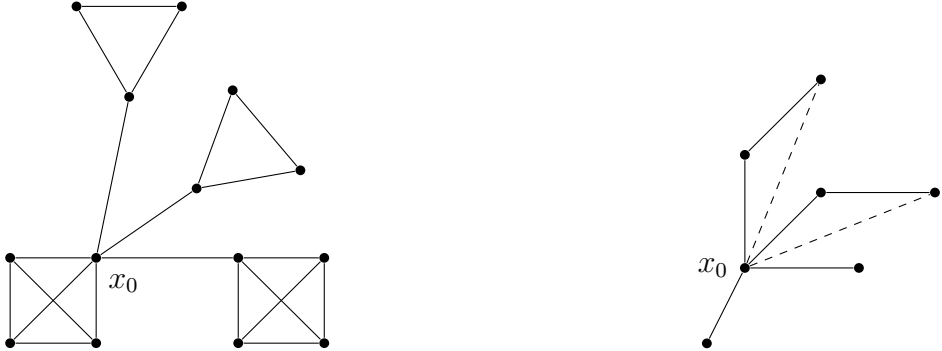
\begin{figure}[H]
\begin{minipage}[t]{0.5\textwidth}
\begin{center}
\begin{tikzpicture}

\def\krad{0.8}

\def\angle{360/3}
\begin{scope}[shift={(1,3.5)}];
\node[fill, circle, inner sep=1.3pt] (w1) at (270-\angle/2+0*\angle+60:\krad){};
\node[fill, circle, inner sep=1.3pt] (w2) at (270-\angle/2+1*\angle+60:\krad){};
\node[fill, circle, inner sep=1.3pt] (w3) at (270-\angle/2+2*\angle+60:\krad){};
\end{scope}

\begin{scope}[shift={(2.5,2)}];
\node[fill, circle, inner sep=1.3pt] (x1) at (270-\angle/2+0*\angle+10:\krad){};
\node[fill, circle, inner sep=1.3pt] (x2) at (270-\angle/2+1*\angle+10:\krad){};
\node[fill, circle, inner sep=1.3pt] (x3) at (270-\angle/2+2*\angle+10:\krad){};
\end{scope}

\def\angle{360/4}
\begin{scope}[shift={(3,0)}];
\node[fill, circle, inner sep=1.3pt] (y1) at (270-\angle/2+0*\angle:\krad){};
\node[fill, circle, inner sep=1.3pt] (y2) at (270-\angle/2+1*\angle:\krad){};
\node[fill, circle, inner sep=1.3pt] (y3) at (270-\angle/2+2*\angle:\krad){};
\node[fill, circle, inner sep=1.3pt] (y4) at (270-\angle/2+3*\angle:\krad){};
\end{scope}

\begin{scope}[shift={(0,0)}];
\node[fill, circle, inner sep=1.3pt] (v1) at (270-\angle/2+0*\angle:\krad){};
\node[fill, circle, inner sep=1.3pt] (v2) at (270-\angle/2+1*\angle:\krad){};
\node[fill, circle, inner sep=1.3pt, label=285:$x_0$] (v3) at (270-\angle/2+2*\angle:\krad){};
\node[fill, circle, inner sep=1.3pt] (v4) at (270-\angle/2+3*\angle:\krad){};
\end{scope}

\draw (w1) -- (w2);
\draw (w1) -- (w3);
\draw (w2) -- (w3);

\draw (x1) -- (x2);
\draw (x1) -- (x3);
\draw (x2) -- (x3);
  
\draw (y1) -- (y2);
\draw (y1) -- (y3);
\draw (y1) -- (y4);
\draw (y2) -- (y3);
\draw (y2) -- (y4);
\draw (y3) -- (y4);

\draw (v1) -- (v2);
\draw (v1) -- (v3);
\draw (v1) -- (v4);
\draw (v2) -- (v3);
\draw (v2) -- (v4);
\draw (v3) -- (v4);

\draw (v3) -- (w1);
\draw (v3) -- (x1);
\draw (v3) -- (y4);

\end{tikzpicture}
\end{center}
\end{minipage}\begin{minipage}[t]{0.5\textwidth}
\begin{center}
\begin{tikzpicture}

\node[fill, circle, inner sep=1.3pt, label=left:$x_0$] (v1) at (0,0) {};
\node[fill, circle, inner sep=1.3pt] (v2) at (-0.5,-1) {};
\node[fill, circle, inner sep=1.3pt] (v3) at (0,1.5) {};
\node[fill, circle, inner sep=1.3pt] (v4) at (1,2.5) {};
\node[fill, circle, inner sep=1.3pt] (v5) at (1,1) {};
\node[fill, circle, inner sep=1.3pt] (v6) at (2.5,1) {};
\node[fill, circle, inner sep=1.3pt] (v7) at (1.5,0) {};

\draw (v1) --(v2);
\draw (v1) --(v3);
\draw (v1) --(v5);
\draw (v1) --(v7);
\draw (v3) --(v4);
\draw (v5) --(v6);

\draw[dashed] (v1) --(v4);
\draw[dashed] (v1) --(v6);

\end{tikzpicture}
\end{center}
\end{minipage}
\caption{The graph $F(14,4)$ on the left and a member of $\mathcal{F}(7,4)$ on the right, where dashed lines are potential edges.}\label{fig1}
\end{figure}
The graph $G(n,\alpha)$ has exactly 
$\alpha-n\,{\rm mod}\,\alpha$ components of order
$\left\lfloor\frac{n}{\alpha}\right\rfloor$,
which implies
$$
\sharp\alpha(G(n,\alpha))=
g(n,\alpha):=
\left\lfloor\frac{n}{\alpha}\right\rfloor^{\alpha-(n\,{\rm mod}\,\alpha)}
\left\lceil\frac{n}{\alpha}\right\rceil^{n\,{\rm mod}\,\alpha}.$$
For $\frac{n}{\alpha}\geq 2$, we have that
$F(n,\alpha)-x_0$ is isomorphic to $G(n-1,\alpha)$, which implies 
\begin{eqnarray*}
\sharp\alpha(F(n,\alpha)) & = & f(n,\alpha)
:= g(n-1,\alpha)+
\left(\left\lfloor\frac{n}{\alpha}\right\rfloor-1\right)^{\alpha-n\,{\rm mod}\,\alpha}\left(\left\lceil\frac{n}{\alpha}\right\rceil-1\right)^{n\,{\rm mod}\,\alpha-1},
\end{eqnarray*}
where the term added to $g(n-1,\alpha)$ counts the maximum independent sets in $F(n,\alpha)$ that contain $x_0$.
For $\frac{n}{\alpha}<2$, 
the added term evaluates to $0$,
that is, $f(n,\alpha)$ equals $g(n-1,\alpha)$.
Furthermore, we obtain $\alpha\leq n-1\leq 2\alpha-2$, 
which implies that $G(n-1,\alpha)$ has 
$$\alpha- (n-1)\,{\rm mod}\,\alpha=\alpha-(n-1-\alpha)=2\alpha-n+1\geq 2$$ isolated vertices.
This implies that the vertex $x_0$ whose removal from a graph $G$ 
in ${\cal F}(n,\alpha)$ yields $G(n-1,\alpha)$ does not belong to 
any maximum independent set in $G$
for $\frac{n}{\alpha}<2$.
Hence, also in this case, we obtain
$$\sharp\alpha(G) =f(n,\alpha)= g(n-1,\alpha)$$
for every graph $G$ in ${\cal F}(n,\alpha)$.

Note that 
$$\sharp\alpha(C_5)=\sharp\alpha(F(5,2))=f(5,2)=5.$$
The following result is 
an immediate consequence of Zykov's generalization \cite{zy} 
of Tur\'{a}n's theorem \cite{tu}; see \cite{mora} for a simple proof.

\begin{theorem}\label{theorem1}
If $G$ is a graph of order $n$ and independence number $\alpha$ 
with $\alpha<n$, then
$\sharp\alpha(G)\leq g(n,\alpha)$
with equality if and only if $G$ is isomorphic to $G(n,\alpha)$.
\end{theorem}
Our contribution in the present paper is 
the following connected version of Theorem \ref{theorem1},
which was recently conjectured by Derikvand and Oboudi \cite{deob}.

\begin{theorem}\label{theorem2}
If $G$ is a connected graph of order $n$ 
and independence number $\alpha$ 
with $\alpha<n$, then
$\sharp\alpha(G)\leq f(n,\alpha)$
with equality if and only if $G$ is isomorphic to a graph in ${\cal F}(n,\alpha)$.
\end{theorem}
In \cite{deob}, 
Derikvand and Oboudi verify Theorem \ref{theorem2}
for $\alpha\in\{1,2,n-3,n-2,n-1\}$,
that is, for very small and very large values of the independence number.
The maximum number of maximum/maximal independent sets has been studied in some further classes of graphs, 
and we refer the reader to \cite{joch,mora,sava,zi}.

The rest of the paper is devoted to the proof of our main result.

\section{Proof of Theorem \ref{theorem2}}

We begin with two preparatory lemmas.

\begin{lemma}\label{lemma3}
Let $G$ be a connected graph 
of order $n$ and independence number $\alpha$ with $\alpha<n$.
If some vertex $u$ of $G$ is contained 
in no maximum independent set in $G$, 
then $\sharp\alpha(G)\leq f(n,\alpha)$
with equality if and only if 
$G\in \mathcal{F}(n,\alpha)$.
\end{lemma}
\begin{proof}
By the hypothesis and Theorem \ref{theorem1},
we obtain $\sharp\alpha(G)=\sharp\alpha(G-u)\leq g(n-1,\alpha)$ 
with equality if and only if $G-u$ is isomorphic to $G(n-1,\alpha)$.
It follows that 
$\sharp\alpha(G)\leq g(n-1,\alpha)
\leq f(n,\alpha)$,
and that $\sharp\alpha(G)=f(n,\alpha)$ holds
if and only if 
$G-u$ is isomorphic to $G(n-1,\alpha)$,
and
$g(n-1,\alpha)=f(n,\alpha)$.
Since $g(n-1,\alpha)=f(n,\alpha)$ implies $\frac{n}{\alpha}<2$,
the definition of ${\cal F}(n,\alpha)$ for $\frac{n}{\alpha}<2$ implies that 
$\sharp\alpha(G)=f(n,\alpha)$
holds if and only if $G\in\mathcal{F}(n,\alpha)$. 
\end{proof}
The second lemma concerns graphs 
whose structure is similar to 
the structure of the graphs in ${\cal F}(n,\alpha)$.

\begin{lemma}\label{lemma2}
Let $n$ and $\alpha$ be positive integers with $\alpha<n$.
\begin{enumerate}[(i)]
\item Let $G$ be a connected graph 
of order $n$ and independence number $\alpha$, 
whose vertex set is the disjoint union of 
the vertex sets of $\alpha$ cliques $C_0,\ldots,C_{\alpha-1}$.
Let all edges of $G$ 
that do not lie in one of these cliques 
be incident with a vertex $x_0$ in $C_0$,
and let $x_0$ have exactly one neighbor 
in each of the cliques $C_1,\ldots,C_{\alpha-1}$.

Under these assumptions 
$\sharp\alpha(G)\leq f(n,\alpha)$
with equality if and only if 
$G$ is isomorphic to $F(n,\alpha)$.
\item Let the graph $G'$ arise from $F(n,\alpha)$
by adding an edge $uv$ between two non-adjacent vertices of $F(n,\alpha)$. 
\subitem If $\frac{n}{\alpha}\geq 2$, 
then $\alpha(G')=\alpha$ and $\sharp\alpha(G')<f(n,\alpha)$,
and,
\subitem if $\frac{n}{\alpha}<2$, and $u$ and $v$ are distinct from the special cutvertex $x_0$ of $F(n,\alpha)$,
then  
\subsubitem either $\alpha(G')<\alpha$ 
\subsubitem or $\alpha(G')=\alpha$ and $\sharp\alpha(G')<f(n,\alpha)$.
\end{enumerate}
\end{lemma}
\begin{proof} (i)
If $C_i$ has order $n_i$ for $i$ in $\{ 0,\ldots,\alpha-1\}$,
then
\begin{eqnarray*}
\sharp\alpha(G)&=&(n_0-1)\prod_{k=1}^{\alpha-1} n_k+\prod_{k=1}^{\alpha-1} (n_k-1)=\prod_{k=0}^{\alpha-1} n_k-\prod_{k=1}^{\alpha-1} n_k+\prod_{k=1}^{\alpha-1} (n_k-1).
\end{eqnarray*}
In view of the desired statement, 
we may assume that the $n_i$ are such that 
$\sharp\alpha(G)$ is as large as possible.
By symmetry, we may assume 
$n_1\geq\ldots\geq n_{\alpha-1}$.
In order to complete the proof,
it suffices to show that $n_0\geq n_1$ and $n_{\alpha-1}\geq n_0-1$.

If $n_i=1$ for some $i\in\{ 0,\alpha-1\}$,
then every maximum independent set in $G$ 
contains the unique vertex, say $u$, in $C_i$.
It follows that some neighbor, say $v$, of $u$
belongs to no maximum independent set in $G$,
and Lemma \ref{lemma3} implies the desired statement.
Hence, we may assume $n_0,n_{\alpha-1}\geq 2$.

First, we suppose that that $n_0+1\leq n_1$. 
Moving one vertex from $C_1$ to $C_0$ 
results in a graph $G'$ 
of order $n$ and independence number $\alpha$
with 
\begin{eqnarray*}
\sharp\alpha(G')&=&n_0\left(\frac{n_1-1}{n_1}\right)\prod_{k=1}^{\alpha-1} n_k+\left(\frac{n_1-2}{n_1-1}\right)\prod_{k=1}^{\alpha-1} (n_k-1)\\
&=&\prod_{k=0}^{\alpha-1} n_k-n_0\prod_{k=2}^{\alpha-1} n_k+\prod_{k=1}^{\alpha-1} (n_k-1)-\prod_{k=2}^{\alpha-1} (n_k-1).\end{eqnarray*}
Since 
\begin{eqnarray*}
\sharp\alpha(G')-\sharp\alpha(G)&=&\prod_{k=1}^{\alpha-1} n_k-n_0\prod_{k=2}^{\alpha-1} n_k-\prod_{k=2}^{\alpha-1} (n_k-1)\\
&=&(n_1-n_0)\prod_{k=2}^{\alpha-1} n_k-\prod_{k=2}^{\alpha-1} (n_k-1)\\ 
&>&0,
\end{eqnarray*}
we obtain a contradiction to the choice of the $n_i$.

Next, we suppose that $n_{\alpha-1}\leq n_0-2$. 
Moving a vertex from $C_0$ to $C_{\alpha-1}$ 
results in a graph $G'$ 
of order $n$ and independence number $\alpha$
with 
\begin{eqnarray*}
\sharp\alpha(G')&=&(n_0-1)\frac{(n_0-2)(n_{\alpha-1}+1)}{(n_0-1)n_{\alpha-1}}\prod_{k=1}^{\alpha-1} n_k+\left(\frac{n_{\alpha-1}}{n_{\alpha-1}-1}\right)\prod_{k=1}^{\alpha-1} (n_k-1).
\end{eqnarray*}
Since $\frac{(n_0-2)(n_{\alpha-1}+1)}{(n_0-1)n_{\alpha-1}}>1$
and $\frac{n_{\alpha-1}}{n_{\alpha-1}-1}>1$,
we obtain 
$\sharp\alpha(G')>\sharp\alpha(G)$,
which is a contradiction to the choice of the $n_i$,
and completes the proof of (i).\\[3mm]
(ii) We leave the simple proof of this to the reader.
\end{proof}
We proceed to the proof of our main result.

\begin{proof}[Proof of Theorem \ref{theorem2}]
Suppose, for a contradiction, that the theorem fails,
and that $n$ is the smallest order 
of a counterexample $G_0$,
which has independence number $\alpha$.
Since the result is easily verified for $n\leq 5$ or $\alpha=1$,
we may assume that $n\geq 6$ and $\alpha\geq 2$.
Furthermore, we may assume that the connected graph $G_0$
maximizes $\sharp\alpha(G_0)$
among all connected graphs of order $n$ 
and independence number $\alpha$.
Since $G_0$ is a counterexample, 
we have 
\begin{itemize}
\item either $\sharp\alpha(G_0)> f(n,\alpha)$
\item or $\sharp\alpha(G_0)=f(n,\alpha)$
but $G_0\not\in\mathcal{F}(n,\alpha)$.
\end{itemize}
For the rest of the proof,
let the vertex $x$ of $G_0$ maximize $\sharp\alpha(G_0,x)$,
that is, $x$ is contained in the maximum number 
of maximum independent sets in $G_0$. 
Let the set $N$ be the closed neighborhood $N_{G_0}[x]$ of $x$ in $G_0$.

Applying the so-called {\it Moon-Moser operation},
we recursively construct a finite sequence of graphs
$$G_0,\ldots,G_k$$
such that, for every $i\in\{ 0,1,\ldots,k\}$,
\begin{itemize}
\item $G_i$ is a connected graph with vertex set $V(G_0)$,
\item $N_{G_i}[x]=N$,
\item $G_i$ has independence number $\alpha$,
\item $\sharp\alpha(G_i)=\sharp\alpha(G_0)$, and
\item $\sharp\alpha(G_0,x)=\sharp\alpha(G_i,x)\geq \sharp\alpha(G_i,u)$
for every vertex $u\in N$.
\end{itemize}
Trivially, $G_0$ has all these properties.

Now, suppose that $G_{i-1}$ has been constructed 
for some positive integer $i$, 
and that $N$ contains a vertex $y_i$ such that
$y_i$ is not a cutvertex of $G_{i-1}$,
and $N_{G_{i-1}}[y_i]\not=N$.
In this case, we construct a further graph $G_i$ in the sequence
by removing all edges incident with $y_i$ in $G_{i-1}$,
and adding new edges between $y_i$ 
and all vertices of $N\setminus \{ y_i\}$, 
that is, we turn $y_i$ into a {\it true twin} of $x$.
If no such vertex exists, the sequence terminates with $G_{i-1}$.

Since $G_{i-1}$ is connected, and $y_i$ is not a cutvertex of $G_{i-1}$,
the graph $G_i$ is connected.
By construction, $N_{G_i}[x]=N_{G_{i-1}}[x]=N$.
Since a maximum independent set in $G_{i-1}$ that contains $x$
is also an independent set in $G_i$, we have $\alpha(G_i)\geq \alpha$.
If some independent set $I$ in $G_i$ contains more than $\alpha$
vertices, then $I$ necessarily contains $y_i$, 
and no other vertex from $N=N_{G_i}[y_i]=N_{G_{i-1}}[x]$,
which implies the contradiction that $(I\setminus \{ y_i\})\cup \{ x\}$
is an independent set in $G_{i-1}$ with more than $\alpha$ elements.
Hence, $G_i$ has independence number $\alpha$.
By construction, 
$$\sharp\alpha(G_i)=\sharp\alpha(G_{i-1})
-\sharp\alpha(G_{i-1},y_i)
+\sharp\alpha(G_{i-1},x)
\geq \sharp\alpha(G_{i-1})=\sharp\alpha(G_0),$$
and the choice of $G_0$ implies 
$\sharp\alpha(G_i)=\sharp\alpha(G_0)$.
Similarly, by construction, 
$$\sharp\alpha(G_i,x)=\sharp\alpha(G_{i-1},x)=\sharp\alpha(G_0,x)
\,\,\,\,\,\,\,\mbox{ and }\,\,\,\,\,\,\,
\sharp\alpha(G_i,y_i)=\sharp\alpha(G_i,x).$$
Now, let $u\in N\setminus \{ x,y_i\}$.
Since every independent set in $G_i$ that contains $u$
does not contain $y_i$, 
it is also an independent set in $G_{i-1}$,
which implies 
$$\sharp\alpha(G_i,u)\leq \sharp\alpha(G_{i-1},u)
\leq 
\sharp\alpha(G_{i-1},x)=\sharp\alpha(G_0,x)=\sharp\alpha(G_i,x).$$
Altogether, we established the desired properties for $G_i$.

The final graph in the sequence $G_k$ 
has the additional property that 
$N_{G_k}[y]=N$
for every vertex $y$ in $N$
that is not a cutvertex of $G_k$.
Let the graph $G$ arise from $G_k$ by removing
iteratively as long as possible
one by one 
edges between $N$ and $V(G_0)\setminus N$ 
such that the resulting graph remains connected, 
and still has independence number $\alpha$.
Since the independence number does not change,
we obtain $\sharp\alpha(G)\geq \sharp\alpha(G_k)=\sharp\alpha(G_0)$,
and the choice of $G_0$ implies 
$$\sharp\alpha(G)=\sharp\alpha(G_0),$$
that is, the removal of the edges in $E(G_0)\setminus E(G)$
does not lead to any new maximum independent set.

\begin{claim}\label{clm0}
$G$ is isomorphic to a graph in ${\cal F}(n,\alpha)$.
\end{claim}
\begin{proof}[Proof of Claim \ref{clm0}]
If some vertex of $G$ is contained in no maximum independent set in $G$, then, by Lemma \ref{lemma3},
$f(n,\alpha)\geq \sharp\alpha(G)=\sharp\alpha(G_0)\geq f(n,\alpha)$,
which implies $\sharp\alpha(G)=f(n,\alpha)$.
Again by Lemma \ref{lemma3},
we obtain $G\in {\cal F}(n,\alpha)$.
Hence, we may assume that 
$$\mbox{\it every vertex of $G$ 
belongs to some maximum independent set in $G$.}$$
Let $B$ be the set of cutvertices of $G$ in $N$.
Note that the set $N\setminus B$ contains $x$,
and that all vertices in $N\setminus B$ are true twins of $x$.
Since $G$ is connected, $x$ is contained in some maximum independent set in $G$,
and $\alpha\geq 2$,
the set $B$ is not empty.
A component $C$ of $G-N$ 
for which only one vertex $y$ in $B$ has neighbors in $V(C)$
is a {\it private component of $y$}.
Since every vertex in $B$ is a cutvertex, 
every such vertex has at least one private component.

In order to complete the proof of Claim \ref{clm0},
we insert two further claims.

\begin{claim}\label{clm1}
There is a vertex $y$ in $B$, and a private component $C$ of $y$
such that $C$ has order at least $2$,
and $y$ has exactly one neighbor in $V(C)$.
\end{claim}
\begin{proof}[Proof of Claim \ref{clm1}]
First, we assume that there is a vertex $y$ in $B$ 
as well as a private component $C$ of $y$ 
such that $C$ has order at least $2$.
In view of the desired statement, we may assume 
that $y$ has more than one neighbor in $V(C)$.
Let $z$ be a neighbor of $y$ in $V(C)$.
Since $yz$ is not a bridge in $G$, 
the construction of $G$ implies that 
$G-yz$ has an independent set $I$ of order $\alpha+1$.
Clearly, the set $I$ contains $y$ and $z$.
If $y$ is the only vertex of $B$ in $I$,
than $(I\setminus \{ y\})\cup \{ x\}$
is an independent set in $G$ of order $\alpha+1$,
which is a contradiction.
Hence, $I$ contains more than one vertex from $B$.
If, for every vertex $y'$ in $(I\cap B)\setminus \{ y\}$,
there is some private component $C'$ of $y'$ such that 
$|I\cap C'|<\alpha(C')$,
then the union of $\{ x\}$ 
and maximum independent sets in the components of $G-N$
is an independent set in $G$
that is at least as large as $I$,
which is a contradiction.
Hence,
there is some vertex $y'$ in $(I\cap B)\setminus \{ y\}$
such that $|I\cap C'|=\alpha(C')$
for every private component $C'$ of $y'$.
Let $C'$ be a private component of $y'$.
Since $y'\in I$ and $I$ intersects $V(C')$,
the component $C'$ has order at least $2$.
Since $y'\in I$, and $|I\cap C'|=\alpha(C')$, 
the removal of an edge between $y'$ and a vertex in $C'$
does not increase the independence number.
Therefore, by the construction of $G$,
the vertex $y'$ has exactly one neighbor in $C'$,
and the desired statement follows
for $y'$ and $C'$.

Next, we assume that all private components 
have order exactly $1$.
Since every vertex of $G$ belongs to some maximum independent set in $G$,
there is a maximum independent set $I$ in $G$
that intersects $B$.
Now, if $I$ contains a vertex $y$ from $B$,
then $I$ contains no vertex from any private component of $y$.
Therefore, removing from $I$ all vertices from $B$,
and adding $x$ as well as all vertices of all private components
yields an independent set in $G$ that is larger than $I$,
which is a contradiction.
This completes the proof of Claim \ref{clm1}.
\end{proof}
For the rest of the proof, 
let $y\in B$, and a private component $C$ of $y$ 
be as in Claim \ref{clm1}.

Let $z$ be the unique neighbor of $y$ in $C$.

\begin{claim}\label{clm2}
The graph $G$ has a cutvertex $y'$ such that 
\begin{itemize}
\item $G-y'$ has exactly two components $C'$ and $C''$,
\item $C'$ is a clique,
\item $y'$ is adjacent to every vertex of $C'$, and 
\item $y'$ has exactly one neighbor in $C''$.
\end{itemize}
\end{claim}
\begin{proof}[Proof of Claim \ref{clm2}]
If $\alpha(C)=1$, then $y'=z$ has the desired properties.
Hence, we may assume that $\alpha(C)\geq 2$.

First, we assume that 
$\alpha(C)+\alpha(G-V(C))>\alpha$,
which implies that 
every maximum independent set in $G$
contains either $y$ or $z$, but, trivially, not both.
Since $y$ and $z$ both have degree at least $2$,
and $g(n,\alpha)$ is increasing in $n$, we obtain
\begin{eqnarray*}
\sharp\alpha(G)&=&
\sharp\alpha(G-N_G[y])
+\sharp\alpha(G-N_G[z])
\leq 2g(n-3,\alpha-1).
\end{eqnarray*}
Let the connected graph $G'$ 
of order $n$ and independence number $\alpha$
arise from $G(n-3,\alpha-1)$
by adding a clique $K$ of order $3$,
and edges between one vertex in $K$ 
and one vertex in each component of $G(n-3,\alpha-1)$.
If $\frac{n-3}{\alpha-1}\geq 2$,
then every component of $G(n-3,\alpha-1)$ has order at least $2$,
which implies that $G'$ has strictly more than $2g(n-3,\alpha-1)$ 
maximum independent sets.
In this case,
Lemma \ref{lemma2} implies the contradiction
$$\sharp\alpha(G)\leq 2g(n-3,\alpha-1)<
\sharp\alpha(G')\leq f(n,\alpha).$$
If $\frac{n-3}{\alpha-1}<2$,
then $\sharp\alpha(G')=2g(n-3,\alpha-1)$,
because one component of $G(n-3,\alpha-1)$ has order $1$.
Since $K$ has order $3$, 
Lemma \ref{lemma2} implies 
$\sharp\alpha(G')<f(n,\alpha)$,
that is, also in this case
we obtain the contradiction
$$\sharp\alpha(G)<f(n,\alpha).$$
Hence, we may assume that
$\alpha(C)+\alpha(G-V(C))=\alpha$.

Let $I_y$ and $I_z$ be maximum independent sets in $G$
that contain $y$ and $z$, respectively.
Clearly,
\begin{eqnarray*}
|I_y\cap V(C)|&\leq& \alpha(C),\\
|I_z\cap V(C)|&\leq& \alpha(C),\\
|I_y\cap (V(G)\setminus V(C))|&\leq& \alpha(G-V(C)),\mbox{ and}\\
|I_z\cap (V(G)\setminus V(C))|&\leq& \alpha(G-V(C)).
\end{eqnarray*}
Since $|I_y|=|I_z|=\alpha=\alpha(C)+\alpha(G-V(C))$,
these four inequalities all hold with equality,
that is,
$C$ has a maximum independent set containing $z$ and another one not containing $z$,
and 
$G-V(C)$ has a maximum independent set containing $y$ and another one not containing $y$.

By Theorem \ref{theorem1}, and the choice of $n$,
we obtain
\begin{eqnarray}
\alpha(C-z)&=&\alpha(C)\nonumber\\
\alpha\big(G-(V(C)\cup \{ y\})\big)&=&\alpha(G-V(C))\nonumber\\
1\leq \sharp\alpha(C-z)&<&\sharp\alpha(C),\label{e1}\\
\sharp\alpha\big(G-(V(C)\cup \{ y\})\big)&\leq& 
g(n-n(C)-1,\alpha-\alpha(C)),\label{e2}\mbox{ and}\\
\sharp\alpha(G-V(C))
&=&
\sharp\alpha\big(G-(V(C)\cup \{ y\})\big)
+
\sharp\alpha(G-V(C),y)\nonumber \\
&\leq &f(n-n(C),\alpha-\alpha(C))\label{e3}.
\end{eqnarray}
By (\ref{e1}),
the linear program
$$
\begin{array}{lrcl}
\max & \sharp\alpha(C-z)\cdot r+\sharp\alpha(C)\cdot s & &\\
\mbox{such that } & s & \leq & g(n-n(C)-1,\alpha-\alpha(C))\\
 & r+s & \leq & f(n-n(C),\alpha-\alpha(C))\\
 & r,s & \geq & 0
\end{array}
$$
has the unique optimal solution 
\begin{eqnarray*}
r&=&f(n-n(C),\alpha-\alpha(C))-g(n-n(C)-1,\alpha-\alpha(C))\mbox{ and }\\
s&=&g(n-n(C)-1,\alpha-\alpha(C))\big).
\end{eqnarray*}
Since $x$ belongs to some maximum independent set in $G$,
we have $\alpha(G-y)=\alpha(G)$, and
using (\ref{e2}) and (\ref{e3})
as well as the unique optimal solution of the above linear program, 
we obtain
\begin{eqnarray}
\sharp\alpha(G)
&=&\sharp\alpha(G,y)+\sharp\alpha(G-y)\nonumber\\
&=&
\sharp\alpha(C-z)\cdot \sharp\alpha(G-V(C),y)
+\sharp\alpha(C)\cdot \sharp\alpha\big(G-(V(C)\cup \{ y\})\big)\nonumber\\
& \leq &
\sharp\alpha(C-z)\cdot \Big(f(n-n(C),\alpha-\alpha(C))-g(n-n(C)-1,\alpha-\alpha(C))\Big)
\label{e4}\\
&&+\sharp\alpha(C)\cdot g(n-n(C)-1,\alpha-\alpha(C)),\nonumber
\end{eqnarray}
with equality in (\ref{e4}) if and only if 
(\ref{e2}) and (\ref{e3})
hold with equality.
By Theorem \ref{theorem1}, and the choice of $n$,
this implies that (\ref{e4}) holds with equality
if and only if
\begin{enumerate}[(i)]
\item $G-(V(C)\cup \{ y\})$ is isomorphic to 
$G(n-n(C)-1,\alpha-\alpha(C))$,
and 
\item $G-V(C)$ is isomorphic to a graph in  
${\cal F}(n-n(C),\alpha-\alpha(C))$.
\end{enumerate}
If (i) or (ii) fails, then (\ref{e4}) is a strict inequality.
In this case, 
replacing $G-V(C)$ within $G$ 
by $F(n-n(C),\alpha-\alpha(C))$,
and adding a bridge between the special cutvertex $x_0$ 
of $F(n-n(C),\alpha-\alpha(C))$
and the vertex $z$ of $C$, 
yields a connected graph $G'$ of order $n$ 
and independence number $\alpha$
such that $\sharp\alpha(G')$
equals the right hand side of (\ref{e4}).
Now, 
$\sharp\alpha(G_0)=\sharp\alpha(G)<\sharp\alpha(G')$,
which contradicts the choice of $G_0$.
Altogether, we obtain that (i) and (ii) hold.

If $G-V(C)$ is isomorphic to $C_5$, 
then the neighbor of $x$ distinct from $y$
is neither a cutvertex of $G$ nor a true twin of $x$,
which is a contradiction.
Hence, $G-V(C)$ is not isomorphic to $C_5$.
If $\alpha-\alpha(C)=1$,
then $G-V(C)$ is a clique of order at least $2$,
and $y'=y$ has the desired properties.
Hence, we may assume that $\alpha-\alpha(C)\geq 2$.
By (i) and (ii), the vertex $y$ is the special cutvertex $x_0$ of $G-V(C)$.
If $\frac{n-n(C)}{\alpha-\alpha(C)}<2$, 
then no maximum independent set of $G-V(C)$ contains $y$,
which implies the contradiction 
that no maximum independent set of $G$ contains $y$.
Hence, we may assume that $\frac{n-n(C)}{\alpha-\alpha(C)}\geq 2$.
Now, (ii) implies the existence of a bridge $yy'$ in $G-V(C)$
such that $y'$ has the desired properties.
This completes the proof of Claim \ref{clm2}.
\end{proof}
We are now in a position to complete the proof of Claim \ref{clm0}.

For the rest of the proof, 
let $y'$ and $C'$ be as in Claim \ref{clm2}.

Let $n'=n(C')$, and
let $x'$ be the unique neighbor of $y'$ outside of $C'$.

Since $y'$ and each vertex in $C'$ belongs to some maximum independent set in $G$, we obtain
\begin{eqnarray*}
\alpha(G-y')&=&\alpha,\\
\alpha\big(G-(V(C')\cup\{ y'\})\big) &=& \alpha-1\mbox{, and}\\
\alpha(G-N_G[y'])&=&\alpha-1.
\end{eqnarray*}
Now, 
Theorem \ref{theorem1} and the choice of $n$
imply
\begin{eqnarray}
\sharp\alpha(G)&=&\sharp\alpha(G,y')+\sharp\alpha(G-y')
\nonumber\\
&=&\sharp\alpha(G-N_G[y'])
+n(C')\cdot\sharp\alpha\big(G-(V(C')\cup\{ y'\})\big)
\nonumber\\
&\leq & g(n-n'-2,\alpha-1)
+n'\cdot f(n-n'-1,\alpha-1)\label{e5}.
\end{eqnarray}
By Theorem \ref{theorem1} and Lemma \ref{lemma2},
the right hand side of (\ref{e5})
is an upper bound on the number of maximum independent sets 
of a suitable connected graph of order $n$ and independence number $\alpha$
whose structure is as in Lemma \ref{lemma2}(i).
By Lemma \ref{lemma2}, this implies 
\begin{eqnarray}
g(n-n'-2,\alpha-1)
+n'\cdot f(n-n'-1,\alpha-1)
&\leq & f(n,\alpha).\label{e6}
\end{eqnarray}
Since 
$\sharp\alpha(G)=\sharp\alpha(G_0)\geq f(n,\alpha)$,
it follows that
$\sharp\alpha(G)=f(n,\alpha)$,
and (\ref{e5}) and (\ref{e6}) hold with equality.
We obtain 
\begin{eqnarray*}
\sharp\alpha(G-N_G[y'])&=&g(n-n'-2,\alpha-1)\mbox{ and }\\
\sharp\alpha\big(G-(V(C')\cup\{ y'\})\big)&=&f(n-n'-1,\alpha-1),
\end{eqnarray*}
which, by Theorem \ref{theorem1} and the choice of $n$,
imply that 
\begin{enumerate}[(i)]
\item $G-N_G[y']$ is isomorphic to $G(n-n'-2,\alpha-1)$ and
\item $G-(V(C')\cup\{ y'\})$
is isomorphic to a graph in ${\cal F}(n-n'-1,\alpha-1)$.
\end{enumerate}
By (i), the graph $G-(V(C')\cup\{ y'\})$ can not be isomorphic to $C_5$.

If $\alpha=2$, then $G$ arises by adding a bridge between two disjoint cliques,
and Lemma \ref{lemma2} 
implies that $G$ is isomorphic to a graph in ${\cal F}(n,\alpha)$.

If $\alpha\geq 3$, then (i) and (ii) together imply that 
$x'$ is the special cutvertex $x_0$ of $G-(V(C')\cup\{ y'\})$.
Now, the construction of $G$ from $G_k$, and Lemma \ref{lemma2} 
imply that $G$
is isomorphic to a graph in ${\cal F}(n,\alpha)$.
This complete the proof of Claim \ref{clm0}.
\end{proof}
If $\frac{n}{\alpha}\geq 2$,
then no edge can be added to $G$ 
without reducing $\alpha(G)$ or $\sharp\alpha(G)$,
which implies that $G_k=G$ in this case.
If $\frac{n}{\alpha}<2$,
then the only edges that can be added to $G$ 
without reducing $\alpha(G)$ or $\sharp\alpha(G)$,
are incident with the special cutvertex $x_0$ of $G$.
Altogether, it follows in both cases
that $G_k$ is isomorphic to a graph in ${\cal F}(n,\alpha)$.

Since $G_0$ is a counterexample, we have $k\geq 1$.

First, we assume that $\frac{n}{\alpha}\geq 2$.
This implies that $G_k$ is isomorphic to $F(n,\alpha)$.
Let $C_0,\ldots,C_{\alpha-1}$
and $x_0,\ldots,x_{\alpha-1}$
be as in the definition of $F(n,\alpha)$.
Note that $x$ and $y_k$ are true twins and no cutvertices of $G_k$,
and, hence, belong to the same clique, say $C_i$.
If $C_i\subseteq N_{G_{k-1}}[y_k]$,
then $G_{k-1}$ arises from $G_k$ by adding edges incident with $y_k$,
which implies the contradiction $\sharp\alpha(G_{k-1})<\sharp\alpha(G_k)$.
If $C_j\subseteq N_{G_{k-1}}[y_k]$ for some 
$j\in \{ 0,\ldots,\alpha-1\}\setminus \{ i\}$,
that is, $G_{k-1}$ is a supergraph of a graph as in Lemma \ref{lemma2},
then Lemma \ref{lemma2} implies that $G_{k-1}$ is isomorphic to $F(n,\alpha)$,
which implies the contradiction that $y_k$ is not adjacent to $x$ in $G_{k-1}$.
Since $\alpha(G_k)=\alpha(G_{k-1})$,
the structure of $F(n,\alpha)$ easily implies that 
\begin{eqnarray*}
N_{G_{k-1}}[y_k]\cap C_0&=&C_0\setminus \{ x_0\}\mbox{ and }\\
N_{G_{k-1}}[y_k]\cap C_j&=&C_j\setminus \{ x_j\}
\mbox{ for some $j\in \{ 1,\ldots,\alpha-1\}$ such that $i\in \{ 0,j\}$.}
\end{eqnarray*}
Similarly as in Lemma \ref{lemma2}, we have
$$
\sharp\alpha(G_k)=
(n_0-1)\prod_{k=1}^{\alpha-1} n_k+\prod_{k=1}^{\alpha-1} (n_k-1),$$
where $n_k$ is the order of $C_k$ for $k\in \{ 0,\ldots,\alpha-1\}$.

If $i=0$, then, 
considering the maximum independent sets of $G_{k-1}$ 
that contain neither $x_0$ nor $y_k$,
those that contain $x_0$ but not $y_k$, 
those that contain $y_k$ but not $x_0$, and
that contain $x_0$ and $y_k$,
we obtain 
\begin{eqnarray*}
\sharp\alpha(G_{k-1})&\leq &
(n_0-2)\prod_{k=1}^{\alpha-1} n_k
+\prod_{k=1}^{\alpha-1} (n_k-1)
+\frac{1}{n_j}\prod_{k=1}^{\alpha-1}n_k
+\frac{1}{n_j-1}\prod_{k=1}^{\alpha-1}(n_k-1)\\
& = & \sharp\alpha(G_k)
+\frac{1}{n_j}\prod_{k=1}^{\alpha-1}n_k
+\frac{1}{n_j-1}\prod_{k=1}^{\alpha-1}(n_k-1)
-\prod_{k=1}^{\alpha-1} n_k.
\end{eqnarray*}
Since either $\alpha\geq 3$ and $n_k\geq 2$ for every $k\in \{ 0,\ldots,\alpha-1\}$,
or $\alpha=2$ and $n_1\geq 3$,
this implies the contradiction $\sharp\alpha(G_{k-1})<\sharp\alpha(G_k)$.

If $i=j$, then we obtain 
\begin{eqnarray*}
\sharp\alpha(G_{k-1})&\leq &
(n_0-1)\frac{(n_j-1)}{n_j}\prod_{k=1}^{\alpha-1} n_k
+\frac{(n_j-2)}{(n_j-1)}\prod_{k=1}^{\alpha-1} (n_k-1)
+\frac{1}{n_j}\prod_{k=1}^{\alpha-1}n_k\\
&&+\frac{1}{n_j-1}\prod_{k=1}^{\alpha-1}(n_k-1)\\
& = & \sharp\alpha(G_k)
-\frac{n_0-1}{n_j}\prod_{k=1}^{\alpha-1}n_k
+\frac{1}{n_j}\prod_{k=1}^{\alpha-1}n_k.
\end{eqnarray*}
Since, in this case, we have $x,y_k,x_j\in C_j$,
we obtain $n_0\geq n_j\geq 3$,
which implies the contradiction $\sharp\alpha(G_{k-1})<\sharp\alpha(G_k)$.

Next, we assume that $\frac{n}{\alpha}<2$.
This implies that $G_k$ arises from $F(n,\alpha)$
by adding edges incident with the special cutvertex $x_0$ of $F(n,\alpha)$.
Since $x$ and $y_k$ are true twins and no cutvertices of $G_k$,
was may assume, by symmetry, 
that $C_1=\{ x,y_k\}$, and that $x_0$ is adjacent to $x$ and $y_k$.
Since $y_k$ is a neighbor of $x$ in $G_{k-1}$,
the graph $G_{k-1}$ arises from $F(n,\alpha)$
by adding an edge between $y_k$ and some vertex distinct from $x_0$,
which easily implies the contradiction $\sharp\alpha(G_{k-1})<\sharp\alpha(G_k)$.

This completes the proof.
\end{proof}

\end{document}